\documentclass[twoside]{article}
\usepackage{amsmath,amsthm,amsfonts,amssymb,amscd}
\usepackage{indentfirst}
\def\Om{{\Omega}}

\def\po{{\partial}}
\def\ve{{\varepsilon}}

\def\al{{\alpha}}

\def\vr{{\varphi}}
\def\re{{\mathbb{R}}}
\def\bn{{\mathbb{N}}}
\newcommand{\lv}{\left\vert}
\newcommand{\rv}{\right\vert}
\def\fre{\mathcal}
\def\Om{{\Omega}}

\def\po{{\partial}}

\def\ve{{\varepsilon}}

\def\al{{\alpha}}

\def\bn{{\mathbb{N}}}

\newcommand{\Fin}{\hfill$\Box$}

\newcommand{\disp}{\displaystyle}

\numberwithin{equation}{section}

\newtheorem{proposition}{Proposition}[section]
\newtheorem{theorem}{Theorem}[section]
\newtheorem{remark}{Remark}[section]

\begin{document}

\title{\vspace{-1.3in}
{\sc On the Approximate Controllability of Stackelberg-Nash Strategies for
Linear Heat Equations in $\mathbb{R}^N$ with Potentials}}
\date{}
\maketitle

\vspace{-40pt}

\centerline{\sc Isa\'ias P. de Jesus$^{}$\,\footnote{E-mail: {\tt isaias@ufpi.br}}}

\centerline{Universidade Federal do Piau\'i,\, DM, PI, Brasil}

\centerline{\sc Silvano B. de Menezes$^{}$\,\footnote{Corresponding author:\,\,E-mail: {\tt silvano@mat.ufc.br}}}

\centerline{Universidade Federal do Cear\'a,\, DM, CE, Brasil}

\vspace{20pt}
\begin{abstract} In this paper we establish hierarchic control for the linear
heat equation in $\re^N$ with potentials. Our strategy is inspired by the techniques developed by  J.I. D\'iaz and J.-L. Lions \cite{DL}; however many new difficulties arise  due to lack of compactness of Sobolev embeddings. We manage these adversities by employing similarity variables and weighted Sobolev spaces.

\end{abstract}



\section{Introduction and main result}\label{introd}

Optimization problems appear quite often in a number of problems from engineering sciences and economics. Many mathematical models are formulated in terms of optimization problems involving a unique objective, minimize cost or maximize profit, etc... In
more realistic and involved situations, several (in general conflicting) objectives ought to be considered. In general for the classical mono-objective control problem, a functional adding the objectives of the problem is defined and an unique control is used. When
it is not clear how to average the different objectives or when several controls are used,
it becomes essential to study multi-objective problems.
The different notions of equilibrium  for multi-objective problems were introduced
in economy and games theory (see \cite{N, P, S}).

In this article, motivated by the arguments contained in the work of J.I. D\'iaz and J.-L. Lions \cite{DL}, we investigate similar
question of hierarchic control, employing the Stackelberg-Nash strategy for the semilinear heat equation in an unbounded domain $\disp \Om \subset \mathbb{R}^N$ when the control acts on
the interior of $\Om$.

The general statement of the problem is as follows: Let $\mathcal{O}$, $\mathcal{O}_{1}$,  $\mathcal{O}_{2}$, ..., $\mathcal{O}_{n}$ be a bounded open and
disjoint nonempty subsets of $\Om$. Given a control time $T>0$ and a potentials  $\disp a(x,t)\in L^\infty(\re^N \times (0,T))$ and $\disp b(x,t)\in \left(L^\infty(\re^N \times (0,T))\right)^N$, we consider the following semilinear heat equation:
\begin{equation}\label{1}
 \left| \begin{array}{l}
 u_t - \Delta u +a(x,t)u+b(x,t).\nabla u= f \chi_{\mathcal{O}} +\displaystyle\sum_{i = 1}^{n} w_i \chi_{\mathcal{O}_i}\;\;  \mbox{ in }\; Q=\Om \times (0, T), \\[5pt]\disp
 u=0\;\; \mbox{ on }\; \Sigma=\partial \Om \times (0,T),\\[5pt]\disp
 u(x, 0) =u_0(x)\;\; \mbox{ in }\; \Om,
 \end{array}\right.
\end{equation}
where $u_0,  f$ and $w_i$ are given functions in appropriate spaces, $i=1,2, ..., n$, and $\chi_{\mathcal{O}}$ and $\chi_{\mathcal{O}_i}$ denote the characteristic functions of the sets $\mathcal{O}$ and $\mathcal{O}_i$, respectively. We will assume that $u_0\in L^2(\re^N)$, $ f\in L^2(\re^N \times (0,T))$ and $w_i\in L^2(\re^N \times (0,T))$, $i=1, ..., n$.

The subset $\mathcal{O} \subset\Omega$ is the control domain (which is supposed to be as
small as desired), $\mathcal{O}_{1}$,  $\mathcal{O}_{2}$, ..., $\mathcal{O}_{n}$ the secondary control domains,
the function $f$ is called leader control and $w_i$,  $(i=1,2, ..., n)$, the followers controls.

To avoid repetitions, throughout this paper, the index $i$ will always vary from $1$ to $n$.

\begin{remark}\label{r1} In (\ref{1}) it was assumed $u(x,0) =u_0(x) = 0$ in $\Om$. If $u_0 \not= 0$, the system can be transformed into one equivalent but
with $u_0 = 0$. This is possible in consequence of linearity of the system.
\end{remark}

As the solution $u$ of (\ref{1}) depends on $f, w_1, ..., w_n$ then we denote it by $\disp u=u(x, t, f, \mathbf{w})$, where $\mathbf{w}=(w_1, ..., w_n)$, or sometimes by $\disp u=u(x, t, f, w_1, ..., w_n)$. We will study a control problem for (\ref{1}) in the case of $f$ to be independent and $w_1, . . . ,w_n$ depend on $f$. More explicitly, the control $f$ does a choice for its police and the choices of the polices of $w_1, . . . ,w_n$ depend on that of $f$. For this reason the
control $f$ is called the leader and $w_1, . . . ,w_n$ the followers. This process of control is called by Lions \cite{L1} a hierarchic control.

In order to localize the action of the controls $w_i$, we introduce the functions $\widetilde{\rho}_i(x)$, defined on $\Om$ with real values,
satisfying:
\begin{equation}\label{2}
\left|
\begin{array}{l}
\widetilde{\rho}_i \in L^\infty(\re^N), \;\; \widetilde{\rho}_i \geq 0,\\ \disp
\widetilde{\rho}_i=1 \mbox{ in } \mathcal{G}_i \subset \Om, \mbox{ where $\mathcal{G}_i$ is a region of where $w_i$ works}.
\end{array}
\right.
\end{equation}

The objective is the controls $f, w_1, . . . ,w_n$ work so that the state function $u(x, t)$, solution of the state equation (\ref{1}), reaches, at the time $T$ , an ideal state $u^T(x)$, with cost functionals defined by:
\begin{equation}\label{fsn}
\disp \widetilde{J}_i(f, w_1, \ldots, w_n) =
\displaystyle\frac{1}{2} \displaystyle\int_{0}^{T}
\displaystyle\int_{\mathcal{O}_i} w_i^2 dx dt +
\displaystyle\frac{\alpha_i}{2} \int_{\re^N}\left|  \widetilde{\rho}_i [u(x, T, f, \textbf{w})
- u^T] \right|^2dx,
\end{equation}
with $\textbf{w}=(w_1, ..., w_n)$, $\alpha_i>0$ constant.

We describe the methodology as follows: The followers $w_1, . . . ,w_n$ suppose the leader did a choice $f$ for his police.
Then, they try to find a Nash equilibrium for its costs $\widetilde{J_1}, . . . ,\widetilde{J}_n$ , that is, they look for controls $w_1, . . . ,w_n$ , depending
on $f$, such that they minimize their costs. This means that they look for $w_1, . . . ,w_n$ , depending on $f$, satisfying:
\begin{equation}\label{n1}
\disp \widetilde{J}_i(f, w_1, \ldots, w_i, \ldots, w_n) \leq \widetilde{J}_i(f, w_1, \ldots, \overline{w}_i, \ldots, w_n), \;\;\; \forall\; \overline{w}_i \in L^2(\mathbb{R}^N\times (0,T)).
\end{equation}

The controls $w_1, . . . , w_n$, solutions of the system of $n$ inequalities (\ref{n1}), are called Nash equilibrium for the costs $\widetilde{J}_1, . . . , \widetilde{J}_n$ and they depend on $f$ (cf. Aubin \cite{A}).

\begin{remark}\label{r2} In another way, if the leader makes a choice $f$, then the follower makes also a choice $w_i(f)$ which becomes minimum its cost $\widetilde{J}_i$, that is,
\begin{equation}\label{n1eq}
\disp \widetilde{J}_i(f, w_1, \ldots, w_i, \ldots, w_n) =\disp \inf_{\overline{w}_i\in L^2(\mathbb{R}^N\times (0,T))}\widetilde{J}_i(f, w_1, \ldots, \overline{w}_i, \ldots, w_n).
\end{equation}

This is equivalent to the inequalities (\ref{n1}). This process is called Stackelberg-Nash strategy, see D\'iaz and Lions \cite{DL}.
\end{remark}

The control problem that we will consider is as follows: to find controls $\textbf{w}=(w_1, ..., w_n)$ and the corresponding state $u$ verifying (\ref{1}) and the Nash equilibrium related to the functionals $\widetilde{J}_i$ defined in (\ref{fsn}), subject to the following (approximate) controllability constraint
\begin{equation}\label{3}
\disp u(x, t, f, \mathbf{w}) \in B_{L^2(\re^N)}(u^T, \alpha),
\end{equation}
where $B_{L^2(\re^N)}(u^T, \alpha)$ denote the ball of $L^2(\re^N)$ with center in $u^T$ and ratio $\alpha > 0$ (a given number), i.e. if
\begin{equation}\label{4}
\left|\begin{aligned}
\,& u(x, t, f, \mathbf{w}) \text{ describes a dense subset of the given state space}\\
\,& \text{when $f$ spans the set of all controls available to the leader}.
\end{aligned}
\right.
\end{equation}
\begin{remark}\label{r3}
We emphasize again that in (\ref{4}) the controls $w_i$ are chosen so that (\ref{n1}) is satisfied. Therefore they are functions of $f$.
\end{remark}

To explain this optimal problem, we are going to consider the following two
sub-problems:

\textbf{$\bullet$ Problem 1} Fixed any leader control $f$, to find follower controls $\disp w_1, ..., w_n$ (depending on $f$) and the associated state $u$ solution of (\ref{1}) satisfying the Nash equilibrium related to $\widetilde{J}_i$ defined in (\ref{fsn}).

\textbf{$\bullet$ Problem 2} Assuming that the existence of the Nash equilibrium $w_1(f), . . . ,w_n(f)$ was proved, then when $f$ varies in $L^2(\re^N\times (0,T))$, to prove that the solutions $u(x, t, f, w_1(f), . . . , w_n(f))$ of the state equation (\ref{1}), evaluated at $t = T$, that is, $u(x, T, f, w_1(f), . . . , w_n(f))$, generate a dense subset of $L^2(\re^N)$. This permits to approximate $u^T$.

The problems 1 and 2 were proved by D\'iaz and Lions \cite{DL} when $\Om$ is a bounded set. The aim of this paper is to adapt the techniques introduced in \cite {DL} to unbounded domains by introducing the weighted Sobolev spaces of Escobedo and Kavian \cite{EK} that guarantee the compactness of the Sobolev's embeddings.
Nevertheless, this proof is valid only when $\Om=\mathbb{R}^N$ or a cone-like domain, since it is necessary to make a change of variables. For simplicity we limit ourselves to the case $\Om=\mathbb{R}^N$.
Under these conditions, the semilinear heat equation given in (\ref{1}) reads as follows:
\begin{equation}\label{1.1}
 \left| \begin{array}{l}
 u_t - \Delta u +a(x,t)u+b(x,t).\nabla u= f \chi_{\mathcal{O}} +\displaystyle\sum_{i = 1}^{n} w_i \chi_{\mathcal{O}_i}\;\;  \mbox{ in }\;  \mathbb{R}^N\times (0, T), \\[5pt]\disp
 u(x, 0) =u_0(x)=0\;\; \mbox{ in }\; \mathbb{R}^N.
 \end{array}\right.
\end{equation}

The main result of this paper is the following:
\begin{theorem}\label{mt} Suppose that $f\in L^2(\mathbb{R}^N \times(0,T))$ and there exists a unique Nash equilibrium $w_1, ..., w_n$, depending on $f$, given by the inequalities (\ref{n1}).

Then, the set of solutions $u\left(x, t, f, \mathbf{w}(f)\right)$ of (\ref{1.1}) evaluated at time $t=T$ is dense in $L^2(\re^N)$.

\end{theorem}

In order, to avoid the problems related to the noncompactness of the Sobolev's embeddings
in $\mathbb{R}^n$, we consider the operator
\begin{eqnarray*}
& L\xi = - \Delta \xi - y \displaystyle\frac{\nabla \xi}{2}=-\frac{1}{K}\;\mathrm{div}(K \nabla \xi),\\[5pt]
& \disp K(y) =  \mathrm{exp} \left( \displaystyle\frac{|y|^2}{4} \right),
 \\[5pt]
& \disp D(L) \subset L^2(K) =  \left\{ \xi: \mathbb{R}^N \rightarrow
\mathbb{R} : \displaystyle\int_{\mathbb{R}^N} K(y) |\xi(y)|^2 dy <
\infty \right\},
\end{eqnarray*}
and the evolution equation
\begin{equation}\label{1v}
\left| \begin{array}{l}
 \disp v_s + Lv+A(y,s)v+B(y,s).\nabla v - \frac{N}{2}\;v = g(y,s) \chi_{\mathcal{O}'} + \sum_{i = 1}^{n} h_i(y,s) \chi_{\mathcal{O}'_i} \;\mbox{ in }\; \mathbb{R}^N \times (0,S), \\
 v(y, 0) = v_0(y)  \mbox{ in } \mathbb{R}^N,
 \end{array}\right.
\end{equation}
with $v_0(y) \in L^2(K)$, $\mathcal{O}'$,  $\mathcal{O}'_{i}$ a suitable set in $\mathbb{R}^N$; and $S = \log(T + 1)$. A change
of variables transforms (\ref{1.1}) in (\ref{1v}). In fact, if we define for $s \geq 0$, $y\in \mathbb{R}^N$
\begin{equation}\label{mv1}
\left|
\begin{array}{l}
\disp v(y,s) = e^{\frac{sN}{2}}\, u(e^{\frac{s}{2}}y, e^s - 1)\\[9pt]\disp
\disp A(y,s)=e^{s}\, a(e^{\frac{s}{2}}y, e^s - 1)\\[9pt]\disp
B(y,s)=e^{\frac{s}{2}}\, b(e^{\frac{s}{2}}y, e^s - 1)\\[9pt]\disp
g(y,s) = e^{\frac{s(N+2)}{2}}\,f\big(e^{\frac{s}{2}}y, e^s-1\big)\\[9pt] \disp
h_i(s, y)= e^{\frac{s(N+2)}{2}}\,w_i\big(e^{\frac{s}{2}}y, e^s-1\big),
\end{array}\right.
\end{equation}
then, for $u_0\in L^2(K)$ and $u$ solution of (\ref{1.1}), $v$ verifies (\ref{1v}) with $v_0 = u_0$, $S =
\log(T + 1)$, $\mathcal{O}'(s)=e^{-\frac{s}{2}}\mathcal{O}$, and $\mathcal{O}_i'(s)=e^{-\frac{s}{2}}\mathcal{O}_i$.

Reciprocally, if we know a solution $v$ of (\ref{1v}), and we define for $t \geq 0$, $x \in \mathbb{R}^N$
\begin{equation}\label{mv2}
\left|
\begin{array}{l}
\disp u(x,t) = \left(1+t\right)^{-\frac{N}{2}}\, v(\frac{x}{\sqrt{1+t}}, \log(1+t))\\[9pt]\disp
a(x,t) =\left(1+t\right)^{-1}\, A(\frac{x}{\sqrt{1+t}}, \log(1+t))\\[9pt]\disp
b(x,t) =\left(1+t\right)^{-\frac{1}{2}}\, B(\frac{x}{\sqrt{1+t}}, \log(1+t))\\[9pt]\disp
f(x,t) = \left(1+t\right)^{-\frac{N}{2}-1}\, g(\frac{x}{\sqrt{1+t}}, \log(1+t))\\[9pt] \disp
w_i(x,t) = \left(1+t\right)^{-\frac{N}{2}-1}\, h_i(\frac{x}{\sqrt{1+t}}, \log(1+t)),
\end{array}\right.
\end{equation}
it is not difficult to see that $u$ satisfies (\ref{1.1}) with $u_0 = v_0$ and $T = e^S - 1$.

The change of variables is interesting because the operator L defined above has
compact inverse in $L^2(K)$ and then the equation (\ref{1v}) can be studied in the same
manner as the heat equation in a bounded region $\Om$ of $\mathbb{R}^N$.

The rest of the paper is organized as follows: in Section \ref{sec2}, we introduce the similarity variables and the weighted Sobolev spaces; we discuss some a priori estimates of the norm of the solution in these spaces. In Section \ref{sec3} we investigate the approximate controllability proving the density Theorem \ref{mt}. Section \ref{sec4} is concerned with the existence of the Nash equilibrium which is gotten by using the Lax-Milgram Theorem. Finally, for completeness, in Section \ref{sec5} we recall a result concerning the existence and uniqueness of a solution to~\eqref{1} and also similar results for this adjoint system.

\section{Weighted Sobolev spaces}\label{sec2}

In this section we recall some basic facts about the similarity
variables and weighted Sobolev spaces for the heat equation. We
refer to \cite{EK} for further developments and details.

We consider the solution $u=u(x,t)$ of problem (\ref{1.1}). We now
introduce the new space-time variables
\begin{equation}
y = \frac{x}{\sqrt{t+1}}\,; \qquad s = \log(t+1). \label{2.1}
\end{equation}

Then, given $u=u(x,t)$ solution of (\ref{1.1}) we introduce
\begin{equation}
v(y,s) = e^{sN/2}\, u(e^{s/2}y, e^s - 1). \label{2.2}
\end{equation}

It follows that $u$ solves (\ref{1.1}) if and only if $v$ satisfies (\ref{1v}).
In the sequel, we shall analyze the approximate controllability of
the system (\ref{1v}) similarly to Theorem \ref{mt}.

The elliptic operator involved in (\ref{1v}) may also be written as
\begin{equation}
Lv = -\Delta v - \frac{y\cdot\nabla v}{2} = -\frac{1}{K(y)}\,
\text{div}(K(y)\nabla v) \label{2.4}
\end{equation}
where $K = K(y)$ is the Gaussian weight
\begin{equation}
K(y) = \exp\big(|y|^2/4\big). \label{2.5}
\end{equation}

We introduce the weighted $L^p$-spaces:
$$ L^p(K) = \left\{f\colon
|f|_{L^p(K)} = \left[\int_{\re^N} |f(y)|^p\, K(y)dy\right]^{1/p}< \infty\right\}. $$

For $p=2$, \, $L^2(K)$ is a Hilbert space and
the norm $|\,\cdot\,|_{L^2(K)}$ is induced by the inner product
$$ (f,g) = \int_{\re^N} f(y)\;g(y)\;K(y)\,dy. $$

We then define the unbounded operator $L$ on $L^2(K)$ by setting
$$ Lv = -\Delta v -\frac{y\cdot\nabla v}{2}\,, $$
as above, and $D(L) = \{v \in L^2(K) : Lv \in L^2(K)\}$.

Integrating by parts it is easy to see that $$
\int_{\re^N}(Lv)vK(y)\,dy = \int_{\re^N} |\nabla v|^2\,K(y)\,dy.
$$ Therefore it is natural to introduce the weighted $H^1$-space:
$$ H^1(K) = \left\{f \in L^2(K) : \frac{\po f}{\po y_i} \in
L^2(K), \, i=1,2,\dots,N\right\} $$
endowed with the norm
$$
|f|_{H^1(K)} = \left[\int_{\re^N} \big(|f|^2 + |\nabla f|^2\big)K(y)\,dy\right]^{1/2}. $$

In a similar way, for any $s \in \bn$ we may introduce the space
$$ H^s(K) = \{f \in L^2(K) : D^\al f \in L^2(K), \forall\,\al,
|\al| \le s \}. $$

The operator $L$ defined above has compact
inverse in $L^2(K)$ and the equation (\ref{1v}) can be studied in the
same manner as the heat equation in a bounded region $\Om$ of $\re^N$.

In fact, the following properties were proved in \cite{EK}:
\begin{equation}\label{desigualdades}
\left|
\begin{array}{l}
\disp \int_{\re^N}f^2|y|^2\,K(y)\,dy \le c \int_{\re^N} |\nabla f|^2\,K(y)\,dy,
\,\,\, \forall\, f \in H^1(K), \\[9pt] \disp
\text{the embedding } H^1(K) \hookrightarrow L^2(K) \text{ is compact,} \\[9pt] \disp
 \forall f \in H^1(K), \;\frac N2 \int_{\re^N} K(y)|f|^2dy \leq \int_{\re^N} K(y)|\nabla f|^2 dy,\\[9pt]\disp
 v\in H^1(K) \Leftrightarrow K^{\frac 12} v \in H^1(\re^N), \\[9pt]\disp
 \vr_1 = \exp\bigg(\dfrac{-|y|^2}{4}\bigg) \; \mbox{ is a eigenfunction of } L \mbox{ corresponding to } \; \lambda_1 = N/2,\\[9pt]\disp
 \mbox{ the minimum eigenvalue of } L, \mbox{ i.e. } \; L\vr_1 = \dfrac N2\,\vr_1\\[9pt]\disp
L\colon H^1(K) \to (H^1(K))^{-1}\text{ is an isomorphism,}\\[9pt]\disp
D(L) = H^2(K) \qquad\text{and} \\[9pt]\disp
L^{-1}\colon L^2(K) \to L^2(K) \text{ is self-adjoint and
compact.}\\[9pt]\disp
 \mbox{if } N=1,\,\, v \in H^1(K), \mbox{ then } K^{\frac 12}\,v \in L^\infty(\mathbb{R}),\\[9pt]\disp
 \mbox{if } N=2,\,\, H^1(K)\subset L^q(K)\;\;\forall q\geq 2, \mbox{ and } q\leq \infty,\\[9pt]\disp
 \mbox{if } N=3,\,\, H^1(K)\subset L^{2^*}(K)\;\mbox{ with } 2^*=\frac{2N}{N-2}.
\end{array}
\right.
\end{equation}

\begin{remark}\label{rimer} Moreover, we have $L^2(K) \subset L^1(\re^N)$ with continuous embedding. If $v \in L^2(K)$, then
$$
\int_{\re^N}|v|\;dy = \int_{\re^N} \frac{1}{K^{1/2}}\, K^{1/2}\,|v|\;dy \le \bigg(\int_{\re^N} K|v|^2dy\bigg)^{\frac 12}\,\bigg(\int_{\re^N} \frac 1K\;dy\bigg)^{\frac 12} < \infty.
$$
\end{remark}

Given two separable Hilbert spaces $V$ and $H$ such that $V
\subset H$ with continuous embedding, $V$ being dense in $H$, let
us consider the Hilbert space $$ W(0,T,V,H) = \{\xi \in L^2(0,T,V) :
\xi_t \in L^2(0,T,H)\} $$ equipped with the norm
$$|\xi|_{W(0,T,V,H)}^2 = |\xi|_{L^2(0,T,V)}^2 +|\xi_t|_{L^2(0,T,H)}^2 $$
see, for instance \cite{9}.

Let us recall some important and useful facts about the spaces appearing
in this paper:
\begin{equation}\label{4imersoes}
\left|
\begin{array}{l}
\disp W(0,T,H^1(K), H^{-1}(K)) \subset L^2(0,T,L^2(K)) \;\; \mbox{ with compact imbedding;}\\[3pt]\disp
\disp  W(0,T,H^1(K), H^{-1}(K)) \subset C([0,T], L^2(K)) \;\; \mbox{ with continuous imbedding;}\\[3pt]\disp
\disp W(0,T,H^2(K), L^2(K)) \subset L^2(0,T,H^1(K)) \;\; \mbox{ with compact imbedding;}\\[3pt]\disp
\disp W(0,T,H^2(K), L^2(K)) \subset C([0,T], H^1(K)) \;\; \mbox{  with continuous imbedding.}
\end{array}
\right.
\end{equation}

The results of (\ref{desigualdades}) to (\ref{4imersoes}) allow us to prove that if  $g\in L^2(0,S;L^2(K))$, $h_i\in L^2(0,S;L^2(K))$ and  $v_0 \in L^2(K)$, then system (\ref{1v}) admits a unique solution $v\in W(0,S,H^1(K), H^{-1}(K))$ (see Theorem \ref{eul}).

We will consider the following objective functionals $J_1, . . ., J_n$ defined by
 \begin{equation}\label{fp}
 \left|
 \begin{array}{l}
 \disp J_i: L^2(0,S;L^2(K)) \rightarrow \mathbb{R}\\[5pt]\disp
 J_i(g, h_1, \ldots, h_i, \ldots, h_n)\\[5pt]\disp
  = \frac{1}{2} \int_{0}^{S}
\displaystyle\int_{\mathcal{O}'_i}|D_{y,s}| \;h_i^2\;K(y)\; dy ds +
\displaystyle\frac{\alpha_i}{2} \int_{\re^N} |D_{y}|\;\rho_i(y)^2\left|v(y, S, g,\textbf{h} ) -v^{S}(y)\right|^2 K(y)\;dy,
\end{array}
\right.
\end{equation}
where $v^S(y) = \left(1+T\right)^{\frac{N}{2}} u^T(x)$, with $ \rho_i(y)\geq 0$, $ \rho_i(y)=1$ in $\fre{O}_i'$, $\textbf{h}=(h_1, ..., h_n)$, and $D_{y,s}$, $D_y$ denote the determinant of transformation $\disp (x,t)\to (y,s)$ and $\disp x\to y$, respectively. We note that there exist $k_1$, $k_2$, $k_3$, $k_4$ positive real constants such that
\begin{equation}\label{estd}
\left|
\begin{array}{l}
\disp 0 < k_1 < |D_{y,s}| < k_2\;\; \mbox{ for all } \;\;(y,s),\\[3pt]\disp
0 < k_3 < |D_{y}| < k_4\;\; \mbox{ for all }\; y.
\end{array}
\right.
\end{equation}

The Nash equilibrium for $J_i$ are $h_1, . . . , h_n$, which depend on $g$, solution of
\begin{equation}\label{n1p}
\disp J_i(g, h_1, \ldots, h_i, \ldots, h_n) \leq J_i(g, h_1, \ldots, \overline{h}_i, \ldots, h_n), \;\;\; \forall\; \overline{h}_i \in L^2(0,S;L^2(K)).
\end{equation}

From the regularity and uniqueness of the solution $v$ of (\ref{1v}) the cost functionals $J_1, . . . , J_n$ are well defined. Henceforth, as $J_i$ is convex and lower semi-continuous, then $h_1, . . . , h_n$ is a Nash equilibrium if, and only if, it verifies the Euler-Lagrange equation.

From the definition of $J_i$, we obtain
 \begin{equation}\label{2.7}
 \begin{array}{l}
 \disp J_i\left(g, h_1, \ldots, h_i +
\lambda \widehat{h}_i, \ldots, h_n\right)=\frac{1}{2} \int_{0}^{S}
\displaystyle\int_{\mathcal{O}'_i}|D_{y,s}| \left(h_i+\lambda \widehat{h}_i\right)^2K(y)\; dy ds\\[9pt]\disp
 +
\displaystyle\frac{\alpha_i}{2} \int_{\re^N} |D_{y}|\;\rho_i(y)^2\left|v(y,S, g, h_1, ..., h_i+\lambda \widehat{h}_i, ..., h_n) -v^{S}(y)\right|^2 K(y)\;dy
\end{array}
\end{equation}

\begin{remark}\label{re} As the state equation (\ref{1v}) is linear, for any choice of the controls $ g, h_i$, its unique solution at the time $s$ can
be written as
\begin{equation}\label{2.8}
 \begin{array}{l}
 \disp v(s)=Q_0(s) g+\sum_{i=1}^n Q_i(s) h_i, \;\;0\leq s \leq S,
\end{array}
\end{equation}
where $Q_i$ are linear continuous operators depending on the controls. At the time $s = S$ one has
\begin{equation}\label{2.9}
 \begin{array}{l}
 \disp v(S)=L_0 g+\sum_{i=1}^n L_i h_i,
\end{array}
\end{equation}
where $L_i$ are also linear and continuous operators.
\end{remark}

Then, from (\ref{2.7}) and Remark \ref{re}, we have
\begin{equation}\label{2.10}
 \begin{array}{l}
 \disp J_i\left(g, h_1, \ldots, h_i +
\lambda \widehat{h}_i, \ldots, h_n\right)=\frac{1}{2} \int_{0}^{S}
\displaystyle\int_{\mathcal{O}'_i}|D_{y,s}| \left(h_i+\lambda \widehat{h}_i\right)^2K(y)\; dy ds\\[9pt]\disp
 +\displaystyle\frac{\alpha_i}{2} \int_{\re^N} |D_{y}|\;\rho_i(y)^2\left|L_0g+L_1h_1+ ...+L_i( h_i+\lambda \widehat{h}_i)+ ...+L_nh_n) -v^{S}(y)\right|^2 K(y)\;dy.
\end{array}
\end{equation}

The derivative with respect to $\lambda$, evaluated at $\lambda = 0$, is the Gateaux derivative given by

\begin{equation}\label{2.11}
\begin{array}{l}
\displaystyle\frac{d}{d \lambda}\; J_i\left(g, h_1, \ldots, h_i +
\lambda \widehat{h}_i, \ldots, h_n\right)\Bigg|_{\lambda = 0}=\displaystyle\frac 12 \int_{0}^{S} \displaystyle\int_{\mathcal{O}_i'} |D_{y,s}|
  h_i \widehat{h}_i K(y)\; dy ds \\[9pt]\disp
 + \frac {\alpha_i}{2} \displaystyle\int_{\mathbb{R}^N} |D_{y}|\;\rho_i(y)^2 [v(y, S, g,\textbf{h} ) -v^{S}(y)]\; L_i\widehat{h}_i\;K(y)
 dy,
 \end{array}
 \end{equation}
 for all $\widehat{h}_i$ in $L^2(0,S;L^2(K))$, where $\mathbf{h}=(h_1, ..., h_n)$ and $L_i\widehat{h}_i=\widehat{v}_i(y,S;\widehat{h}_i)$ (cf. Remark \ref{re}), with $\widehat{v}_i(y,s;\widehat{h}_i)$ solution of
 \begin{equation}\label{2.12}
 \left| \begin{array}{l}
\displaystyle\frac{\partial \widehat{v}_i}{\partial s} + L
\widehat{v}_i +A(y,s)\widehat{v}_i+B(y,s)\;.\;\nabla \widehat{v}_i- \displaystyle\frac{N}{2}\widehat{v}_i =
\widehat{h}_i \chi_{\mathcal{O}_i'}\;\; \mbox{ in } \;\;
\mathbb{R}^N \times (0,S), \\[9pt]\disp
\widehat{v}_i(0) = 0 \;\;\mbox{ in }\;\; \mathbb{R}^N.
 \end{array} \right.
 \end{equation}

Therefore, $h_1, . . ., h_n$ is a Nash equilibrium for the cost functionals $J_1, . . ., J_n$ if, and only if, it verifies the Euler-Lagrange equation
\begin{equation}\label{2.13}
J'_i(g, h_1, \ldots, h_i, \ldots, h_n).\widehat{h}_i=0, \;\;\mbox{ for all } \widehat{h}_i \in L^2(0,S;L^2(K)).
\end{equation}

Explicitly we have
\begin{equation}\label{2.14}
\begin{array}{l}
\displaystyle \int_{0}^{S} \displaystyle\int_{\mathcal{O}_i'} |D_{y,s}|\;
  h_i \widehat{h}_i K(y)\; dy ds
 + \alpha_i\displaystyle\int_{\mathbb{R}^N} |D_{y}|\;\rho_i^2(y) [v(y, S, g,\textbf{h} ) -v^{S}(y)]\; \widehat{v}_i\;K(y)
 dy=0,
 \end{array}
 \end{equation}
for all $\widehat{h}_i \in L^2(0,S;L^2(K))$ and $\widehat{v}_i(y,s;\widehat{h}_i)$ solution of (\ref{2.12}).

Note that if $h_1, . . ., h_n$ is the unique Nash equilibrium, depending on $g$, for (\ref{fp}), then its transformation, by the (\ref{mv2}), is
the unique Nash equilibrium $w_1, . . ., w_n$ for (\ref{fsn}), which depends on $f$.

We recall that our initial problem was the controls $f$, $w_1$, . . ., $w_n$ work such that the function $u(x, t)$, unique solution
of (\ref{1.1}), reaches a fixed state $u^T (x) \in L^2(\re^N)$, at time $T$, with cost functionals defined by (\ref{fsn}).

From the change of variables (\ref{mv1}), this problem in $\re^N\times (0,T)$ was transformed into one equivalent (\ref{1v}) in some weighted Sobolev space. Thus, for fixed $v^S(y) \in L^2(K)$ the controls $g, h_1, . . ., h_n$ must work such that the unique solution $v(y, s, g, \mathbf{h}(g))$ of (\ref{1v}), evaluated at $s = S$, reaches the ideal state $v^S(y)$. This will be done in the sense of an approximate controllability. In fact, it is sufficient to prove if $h_1, . . . , h_n$, depending of $g$, is the unique Nash equilibrium for the cost functionals (\ref{fp}), then we have an approximate controllability.
This means that if there exists the Nash equilibrium and $v(y, s; g, \mathbf{h}(g))$ is the unique solution of (\ref{1v}), then the set
generated by $v(y, S; g, \mathbf{h}(g))$ is dense in $L^2(K)$, that is, approximate $v^S(y)$. This will be proven in the next section.

\section{Approximate controllability}\label{sec3}

This section is devoted to proving the main result, namely, Theorem \ref{mt}. For this, initially we obtain the following result on the "approximate controllability" for the system (\ref{1v}):
\begin{theorem}\label{mtp}
Suppose $g\in L^2(0,S;L^2(K))$ and there exists a unique Nash equilibrium $h_1(g), ..., h_n(g)$, depending on $g$, given by the inequalities (\ref{n1p}). Then, the set of solutions $v(y,s;g,\mathbf{h}(g))$ of (\ref{1v}) evaluated at time $s=S$ in dense in $L^2(K)$.
\end{theorem}
\begin{proof}
It will be done in two steps. First, we find the adjoint state and the optimality system. In the second one we prove
the approximate controllability by means of a simple argument of functional analysis and results of unique continuation.

The application of this method requires only the unique
continuation property due to C. Fabre \cite{Fa} (Theorem 1.4 of \cite{Fa}).
More precisely, we have:

\vglue .2in

\begin{proposition}\label{prop1} Let $w$ be an open and nonempty
set of $\re^N$, \, $\hat\omega(s) = e^{-s/2}w$ and $\hat q=
\{(y,s); s \in (0,S), y \in \hat\omega(s)\}$. Assume that $A(y,s)
\in L^\infty(\re^N\times(0,S))$, \,\, $B(y,s) \in
(L^\infty(\re^N\times(0,S)))^N$. Let $p \in L^2(0,S;H^1(K))$ be
such that
\begin{equation}\nonumber
\left\vert
\begin{aligned}
\,&-p_s + Lp + Ap - \mbox{div}(Bp) - \frac N2\,p - \frac 12\,
y\cdot Bp=0 \text{ in } \re^N\times(0,S)\\ \,&p=0
\quad\text{in}\quad \hat q.
\end{aligned} \right.
\end{equation}
Then $p \equiv 0$.
\end{proposition}

\vglue .1in

\noindent{\bf Proof of Proposition \ref{prop1} } $L$, $A$ and $B$ do satisfy the conditions
of Theorem 1.4 due C. Fabre \cite{Fa} and the assumption on $p$ implies
$p \in L_{\text{loc}}^2(0,S,H_{\text{loc}}^1(\re^N))$.
Consequently $p \equiv 0$. \qed

\vglue .1in

$\mathbf{Step\; 1.}$ Suppose there exists the Nash equilibrium $h_1, . . ., h_n$ depending on $g$ for the cost functionals $J_1, . . ., J_n$ defined
in (\ref{fp}). Thus, it implies this is a solution of the Euler-Lagrange equation (\ref{2.14}), conditioned to the system of linear parabolic
equations (\ref{2.12}).

In order to express (\ref{2.14}) in a convenient form, we introduce the adjoint state $p_i$ defined by
\begin{equation}\label{sa}
\left|
\begin{array}{l}
\disp - p_i' + L p_i+Ap_i-\mbox{div }(Bp_i) -\frac{N}{2}\; p_ i-\frac 12\; y\;.\;Bp_i = 0 \;\; \mbox{
in } \;\; \mathbb{R}^N \times (0,S), \\[5pt]\disp
 p_i(y,S) = |D_y|\;\rho_i^2(y)\left[ v(y, S; g, \mathbf{h}) -
v^{S}(y) \right] \;\; \mbox{ in }\;\; \mathbb{R}^N.
\end{array} \right.
\end{equation}

At the end of this paper, we recall without proof a (relatively well known) result concerning the adjoint system (\ref{sa}). Under natural hypotheses on $A(y,s)$ and $B(y,s)$, we see there that, for each $p(y,s) \in H^{-1}(K)$, there exists exactly one solution $p_i$ to (\ref{sa}), with $\disp p_i \in L^2(0,S;L^2(K))\cap C\left([0,S];H^{-1}(K)\right)$; see Theorem \ref{ta2}.

If we multiply (\ref{sa}) by $K(y)\widehat{v}$ and if we integrate by parts, we find
\begin{equation}\label{3.1}
\displaystyle\int_{\mathbb{R}^N} |D_y| \;\rho_i^2(y) \left(v(y, S; g, \mathbf{h}) - v^{S}(y)  \right) \widehat{v}_i(S)\;K(y)\; dy =
\displaystyle\int_{0}^{S} \displaystyle\int_{\mathbb{R}^N}
 p_i\; \widehat{h}_i\; \chi_{\mathcal{O}_i'}\;K(y)\; dy ds
\end{equation}
so that (\ref{2.14}) becomes
\begin{equation}\label{3.2}
\displaystyle\int_{0}^{S}\int_{\mathcal{O}_i'}  \left(|D_{y,s}|\;h_i  + \alpha_i p_i\right)\;\widehat{h}_i\;K(y)\; dy ds = 0,
\end{equation}
for all  $ \widehat{h}_i  \in L^2(0,S;L^2(K))$. Then $\disp |D_{y,s}|\;h_i  + \alpha_i p_i=0$ in $\mathcal{O}_i' \times (0,S)$. By hypothesis, we obtain
\begin{equation}\label{3.3}
\displaystyle h_i=-\frac{\alpha_i p_i}{|D_{y,s}|}\;\;\mbox{ in } \;\;\mathcal{O}_i' \times (0,S).
\end{equation}

Thus, if $h_1, . . ., h_n$ is the unique Nash equilibrium for the cost functionals $J_1, . . . , J_n$ , which is associated with the state
equation (\ref{1v}), then we get the optimality system
\begin{equation}\label{3.4}
\left|
\begin{array}{l}
\displaystyle\frac{\partial v}{\partial
s} + Lv +Av+B\;.\; \nabla v- \displaystyle\frac{N}{2} v +
\displaystyle\sum_{i = 1}^{n} \frac{\alpha_i p_i}{|D_{y,s}|}\; \chi_{\mathcal{O}'_i} = g\chi_{\mathcal{O}'} \;\; \mbox{ in } \;\; \mathbb{R}^N \times(0,S),\\[9pt]\disp
- p_i' + L p_i+Ap_i-\mbox{div }(Bp_i) -\frac{N}{2}\; p_i-\frac 12\;y\;.\;Bp_i= 0 \;\; \mbox{ in } \;\; \mathbb{R}^N \times (0,S), \\[9pt]\disp
v(0) = 0, \;\;\; p_i(y,S) =|D_y|\; \rho_i^2(y) \left[ v(y,S; g,\mathbf{ h}) -v^{S}(y) \right]  \;\; \mbox{ in } \;\; \mathbb{R}^N.
\end{array} \right.
\end{equation}

We recall that here we are assuming the existence and uniqueness of a Nash equilibrium . We return to that in Section \ref{sec4}.

$\mathbf{Step\; 2.}$ To prove approximate controllability we assume $v^S (y) = 0$ to simplify the calculus, because the optimality system
(\ref{3.4}) is linear (it suffices to use a translation argument).

Let $\zeta\in L^2(K)$ and let us assume that
\begin{equation}\label{3.5}
\left( v( ., S; g, \mathbf{h}(g)), \zeta \right)_{L^2(K)} = 0,
\end{equation}
for all $g \in L^2(0,S;L^2(K))$. We want to show that $\zeta \equiv 0$.

Motivated by (\ref{3.4}), since we suppose $v^S(y) = 0$, we introduce the solution $\{\varphi, \psi_1, \psi_2,
\ldots, \psi_n \}$ of the adjoint system
\begin{equation}\label{3.6}
\left|
\begin{array}{l}
- \displaystyle\frac{\partial\varphi}{\partial s} + L\varphi +A\varphi-\mbox{div }(B\varphi)-
\displaystyle\frac{N}{2}\; \varphi-\frac 12 \;y\;.\;B\varphi = 0 \;\; \mbox{ in } \;\; \mathbb{R}^N \times (0,S),\\[7pt]\disp
\frac{\partial\psi_i}{\partial s} + L\psi_i +A\psi_i+B\;.\;\nabla \psi_i- \displaystyle\frac{N}{2}\; \psi_i = -\frac{1}{|D_{y,s}|}\alpha_i
\varphi\; \chi_{\mathcal{O}'_i} \;\; \mbox{ in } \mathbb{R}^N \times (0,S),\\[7pt]\disp
\varphi(S) = \zeta + \displaystyle\sum_{i = 1}^{n}|D_y|\;
\psi_i(S)\; \rho_i^2 \;\; \mbox{ in } \;\; \mathbb{R}^N,  \\[7pt]\disp
\psi_i(0) = 0 \;\; \mbox{ in } \;\; \mathbb{R}^N.
\end{array} \right.
\end{equation}

Multiplying $(\ref{3.6})_1$ by $K(y) v$ and $(\ref{3.6})_2$ by  $K(y)p_i$ and integrating in $\re^N\times (0,S)$ we obtain
\begin{equation}\label{3.7}
\begin{array}{l}
-\left( \zeta+\sum_{i =
1}^{n} |D_y|\; \psi_i(S) \rho_i^2, v(S)\right)_{L^2(K)}
 \\[7pt]\disp+ \displaystyle\int_{0}^{S}
\displaystyle\int_{\mathbb{R}^N} \varphi  \Big(
\frac{\partial v}{\partial s} + L v +Av+B\;.\;\nabla v-\frac{N}{2}\;v \Big)K(y)\; dy ds = 0,
\end{array}
\end{equation}
and
\begin{equation}\label{3.8}
\displaystyle\sum_{i = 1}^{n} \left( \psi_i(S),
p_i(S) \right)_{L^2(K)} +
\displaystyle\int_{0}^{S} \int_{\mathbb{R}^N}
\displaystyle\sum_{i = 1}^{n}\frac{1}{|D_{y,s}|} \alpha_i \varphi p_i\;K(y)\; \chi_{\mathcal{O}_i'}\; dy ds = 0.
\end{equation}

From the identity (\ref{3.7}) and (\ref{3.8}) and using (\ref{3.6}), we have
\begin{equation}\label{3.9}
\displaystyle -\left( \zeta, v(S)\right)_{L^2(K)}+\int_{0}^{S}\int_{\mathbb{R}^N} \varphi\; g \; K(y)\;\chi_{\mathcal{O}'} \; dy ds  = 0,\;\;\;\; \mbox{ for all } \;\; g \in L^2(0,S;L^2(K)).
\end{equation}
Therefore, if (\ref{3.5}) holds, this equality implies that
\begin{equation}\label{3.9}
\displaystyle \int_{0}^{S}\int_{\mathcal{O}'} \varphi\; g \; K(y)\; dy ds  = 0,\;\;\;\; \mbox{ for all } \;\; g \in L^2(0,S;L^2(K)).
\end{equation}

It follows by the unique continuation, cf. C. Fabre \cite{Fa}, that $\varphi= 0$ on $\re^N\times (0,S)$. Going back to (\ref{3.6}), the initial-boundary value for $\psi_i$  implies $\psi_i=0$ in $\re^N\times (0,S)$ and from the continuity, so that $(\ref{3.6})_3$ gives $\zeta = 0$. This ends the proof of Theorem \ref{mtp}.
\end{proof}

We conclude this section with the proof of Theorem \ref{mt}, as an immediate consequence of Theorem \ref{mtp}.

\noindent{\bf Proof of Theorem \ref{mt} } Let us assume that $f\in L^2(\mathbb{R}^N \times(0,T))$ and there exists a unique Nash equilibrium $\mathbf{w}=(w_1, ..., w_n)$, depending on $f$, given by the inequalities (\ref{n1}). We need to show that the set of solutions $u(x, t, f, \mathbf{w})$ of (\ref{1.1}) evaluated at time $t=T$ is dense in $L^2(\re^N)$, that is, for any $u_0,u^T \in L^2(\re^N)$, $f\in L^2(\mathbb{R}^N \times(0,T))$ and $\epsilon > 0$, there exists a unique Nash equilibrium $\mathbf{w}=\mathbf{w}(f)$ such that the solution $u$ of (\ref{1.1}) satisfies
\begin{equation}\label{caeu}
|u(T)-u^T|_{L^2(\re^N)} \le \epsilon.
\end{equation}

\begin{remark}\label{obsf}


We recall that, if we know a unique Nash equilibrium $\disp \mathbf{h}=(h_1, . . ., h_n)$, depending on $g$, for (\ref{fp}), then its transformation, by the change of variables (\ref{mv2}), is the unique Nash equilibrium $\disp \mathbf{w}=(w_1, . . ., w_n)$ for (\ref{fsn}), which depends on $f$.

We postpone the proof of the existence and uniqueness of a Nash equilibrium $\disp \mathbf{h}=(h_1, . . ., h_n)$ for (\ref{fp}), however, to Section \ref{sec4}. For the moment, let us assume that these properties hold.

\end{remark}

Our plan to prove Theorem \ref{mt} is divided into two steps.

$\bullet$ First step:\quad The case $u^T \in L^2(K)$.

\noindent We make the change of variables $v^S(y) = (T+1)^{\frac N2}\,
u^T((T+1)^{\frac 12} y) $ and $S = \log(T+1)$. We have that $v^S \in
L^2(K)$ and $v^0 = u^0 =0\in L^2(K)$. From Theorem \ref{mtp} and Remark \ref{obsf} we know the
existence of a unique Nash equilibrium $\disp \mathbf{h}=(h_1, . . ., h_n)$ for (\ref{fp}) such that the solution $v(y, s, g, \mathbf{h})$ of (\ref{1v}) satisfies:
\begin{equation}\label{caeu2}
|v(S)-v^S|_{L^2(K)} \le \epsilon.
\end{equation}

Therefore $$ u(x,t) = (1+t)^{-\frac N2}\,
v\left(\frac{x}{\sqrt{1+t}}\,, \log(1+t)\right) $$ satisfies (\ref{1.1}) with
\begin{equation}\nonumber
\left|
\begin{array}{l}
\disp f(x,t) = \left(1+t\right)^{-\frac{N}{2}-1}\, g(\frac{x}{\sqrt{1+t}}, \log(1+t))\\[9pt] \disp
w_i(x,t) = \left(1+t\right)^{-\frac{N}{2}-1}\, h_i(\frac{x}{\sqrt{1+t}}, \log(1+t)),
\end{array}\right.
\end{equation}

and
\begin{align*}
&|u(T)-u^T|_{L^2(\re^N)}^2 \le \int_{\re^N} (1+T)^{-N}\,
e^{|x|^2\big/4(1+T)} \lv v\left(\frac{x}{\sqrt{1+T}}\,, S\right) -
v^S\left(\frac{x}{\sqrt{1+T}}\right)\rv^2 dx\\ &\le \int_{\re^N}
K(1+T)^{-N/2} |v(y,S) - v^S(y)|^2\,dy \le
 (1+T)^{-N/2} |v(S)-v^S|_{L^2(K)}^2 \le \ve^2.
\end{align*}

\vglue .1in

\noindent Second step:\quad $u^T \in L^2(\re^N)$.

\noindent Since $L^2(K) \subset L^2(\re^N)$ with dense inclusion,
there exists a sequence $\{u_n^T\} \subset L^2(K)$ such that
$u_n^T \to u^T$ strongly in $L^2(\re^N)$. From the first step, we know the existence of a unique Nash equilibrium  $\mathbf{w}_n$ such that $u(x, t, f,\mathbf{w}_n)$ the solution of (\ref{1.1}) with $\mathbf{w}=\mathbf{w}_n$ satisfies
\begin{equation*}
\begin{aligned}
\,&|u(T) - u_n^T|_{L^2(\re^N)} \le \frac{\epsilon}{2}.
\end{aligned}
\end{equation*}

Let $\widetilde N$ be such that for every $n > \widetilde N$,
$$|u^T-u_n^T|_{L^2(\re^N)} \le \frac{\epsilon}{2}. $$

Then $u(x, t, f,\mathbf{w}_{\widetilde N})$ the solution of (\ref{1.1}) with $\mathbf{w}=\mathbf{w}_{\widetilde N}$ satisfies
\begin{equation*}
\begin{aligned}
\,&|u(T)-u^T|_{L^2(\re^N)} \le \epsilon.
\end{aligned}
\end{equation*}

This completes the proof of Theorem \ref{mt}.

\Fin
\section{Existence and uniqueness of a Nash equilibrium}\label{sec4}

Let us consider the cost functionals $J_i$ defined by (\ref{fp}) corresponding to the state equation (\ref{1v}). Firstly, we are going to rewrite the system (\ref{2.14}). For this, we consider the following functional spaces
\begin{equation}\label{4.1}
\left|\begin{array}{l}
\displaystyle \mathcal{H}_i= L^2(0,S;L^2(K)),\\[9pt]\disp
\mathcal{H}=\prod_{i=1}^n \mathcal{H}_i=\left(L^2(0,S;L^2(K))\right)^n,
\end{array}
\right.
\end{equation}
and the (resolvent) operators $L_i\in \mathcal{L}\left(\mathcal{H}_i,L^2(K)\right)$ defined as (see Section \ref{sec2}) $L_ih_i=v_i(S)$, where $v_i(S)$ is $v_i(y,s;h_i)$, evaluated at $s=S$, is the unique solution of the problem:
\begin{equation}\label{4.2}
\left|
\begin{array}{l}
\displaystyle\frac{\partial v_i}{\partial s} + L v_i +Av_i+B\;.\;\nabla v_i- \frac{N}{2} v_i = h_i \chi_{\mathcal{O}'_i} \;\;\; \mbox{ in } \;\; \mathbb{R}^N \times (0,S)\\[7pt]\disp
v_i(0) = 0 \;\; \mbox{ in } \;\; \mathbb{R}^N.
\end{array} \right.
\end{equation}

Note that $v_i$ belongs to $C\left([0,S];H^1(K)\right)$, with
\begin{equation}\label{cis}
\disp |v_i|_{C\left([0,S];H^1(K)\right)}\leq C_i(T)|v_i|_{\mathcal{H}_i},
 \end{equation}

 and $L_ih_i=v_i(S)$ is linear, for each $i$, justifying that $L_i$ belongs in $\mathcal{L}\left(\mathcal{H}_i,H^1(K)\right)$ or in $\mathcal{L}\left(\mathcal{H}_i,L^2(K)\right)$.

Remember that in the definition of $J_i$, given in (\ref{fp}), the argument $v(y, s, g,\textbf{h})$ is the unique solution of the initial-boundary value problem (\ref{1v}) with right-hand side
$$\disp g(y,s) \chi_{\mathcal{O}'} + \sum_{i = 1}^{n} h_i(y,s) \chi_{\mathcal{O}'_i}.$$

Therefore, for fixed $g\in L^2(0,S;L^2(K))$ we obtain $\disp z \in C\left([0, S]; L^2(K)\right)$, as the unique solution of the initial-boundary value problem
\begin{equation}\label{4.3}
\left|
\begin{array}{l}
\displaystyle\frac{\partial z}{\partial s} + Lz+Az+B\;.\;\nabla z -\frac{N}{2}\; z = g(y,s) \chi_{\mathcal {O'}} \;\; \mbox{ in } \;\; \mathbb{R}^N \times (0,S), \\[7pt]\disp
z(x, 0) = 0 \;\; \mbox{ in } \;\; \mathbb{R}^N.
\end{array}
\right.
\end{equation}

Thus, from (\ref{4.2}) and (\ref{4.3}) it follows that $\disp \sum_{i=1}^n v_i+z$ is also a solution of (\ref{1v}). From the uniqueness of the solution
$v(y, s, g,\textbf{h})$ of (\ref{1v}) we can write
$$ \disp v(y, s, g,\textbf{h})=\displaystyle z(y, s, g)+\sum_{i = 1}^{n} v_i(y, s, \textbf{h})=Q_0(s) g+\sum_{i=1}^n Q_i(s) h_i, \;\;0\leq s \leq S.$$

In the last identity above we have used the notation of Remark \ref{re}. Besides, from the continuity it can be evaluated at
$s = S$, which gives
\begin{equation}\label{4.4}
 \disp v(y, S, g,\textbf{h})=\displaystyle z(y, S, g)+\sum_{i = 1}^{n} v_i(y, S, \textbf{h})=z^S(y)+\sum_{i=1}^nL_ih_i(y,S)=L_0 g+\sum_{i=1}^n L_i h_i.
 \end{equation}

Again, in the last identity above we have used the notation of Remark \ref{re}. From (\ref{4.4}), we modify the $J_i$ obtaining
\begin{equation}\label{4.5}
\begin{array}{l}
\disp J_i(g, h_1, \ldots, h_i,\ldots, h_n)\\[9pt]\disp
= \frac 12 \int_0^S\int_{\mathcal{O}_i'}|D_{y,s}|\;h_i^2\;K(y)\;dyds+\frac{\alpha_i}{2} \int_{\re^N} |D_y|\;\rho_i^2(y)\left( \sum_{i=1}^nL_ih_i-\eta^S(y) \right)^2K(y)\;dy   \\[15pt]\disp
=\frac{1}{2} \left| |D_{y,s}|^{\frac 12}h_i
\right|_{\mathcal{H}_i}^{2} + \frac{\alpha_i}{2} \left| |D_y|^{\frac 12} \rho_i
\left( \displaystyle\sum_{i = 1}^{n} L_i h_i - \eta^{S}(y)
\right) \right|_{L^2(K)}^{2},
\end{array}
\end{equation}
with $\eta^{S}(y) = v^{S}(y) - z^{S}(y)$.

It is easy to see that, for each $i$, the functional $J_i$ is convex and lower semi-continuous. Therefore, $\disp \mathbf{h} = (h_1, . . ., h_n) \in \mathcal{H}$ is the Nash equilibrium for the functionals $J_i$ if, and only if, its Gateaux derivative is zero, that is,
\begin{equation}\label{4.6}
\disp \left( |D_{y,s}| h_i,\widehat{h}_i \right)_{\mathcal{H}_i} + \alpha_i \left( \rho_i |D_y|\left(\sum_{j = 1}^{n} L_j h_j - \eta^{S}\right), \rho_i L_i \widehat{h}_i\right)_{L^2(K)} = 0,
\end{equation}
for all $\disp \widehat{h}_i \in L^2(0,S;L^2(K))$.

As the adjoint $L^*_i$ of $L_i$ belongs to $\mathcal{L}\left(L^2(K),\mathcal{H}_i\right)$, then from (\ref{4.6}) one gets
\begin{equation}\label{4.7}
\disp |D_{y,s}|\; h_i + \alpha_i L^*_i \left(\rho_i^2(y) |D_y|\sum_{j = 1}^{n} L_j h_j \right) =\disp \alpha_i
L^*_i \left(\rho_i^2(y)|D_y| \eta^{S}\right),\;\;i=1, ..., n.
\end{equation}

Next, we change the notation to obtain a better formulation of linear system (\ref{4.7}). In fact, setting
$$\disp \xi=(\xi_1, ..., \xi_n),\;\mbox{ with }\; \xi_i= \alpha_i L^*_i \left(\rho_i^2(y)|D_y| \eta^{S}\right),$$

and considering an operator $\mathfrak{L}\in \mathcal{L}(\mathcal{H},\mathcal{H})$ defined by $\mathfrak{L}\mathbf{h}$ with the $n$ components
$$\disp (\mathfrak{L}\mathbf{h})_i=|D_{y,s}|\; h_i + \alpha_i L^*_i \left(\rho_i^2(y) |D_y|\sum_{j = 1}^{n} L_j h_j \right),$$

then the system (\ref{4.7}) is rewritten
\begin{equation}\label{4.8}
\disp \mathfrak{L}\mathbf{h}=\xi \;\; \mbox{ in } \;\mathcal{H}.
\end{equation}

Thus, we must prove that the linear equation (\ref{4.8}) has a solution $\mathbf{h} = (h_1, . . . , h_n)$ in $\mathcal{H}$ for each $\disp \xi=(\xi_1, ..., \xi_n)$ in $\mathcal{H}$. The solvability of (\ref{4.8}) will be established as application of the Lax-Milgram Lemma, with restrictions on $\alpha_i$ and $\rho_i$ .

\begin{theorem}\label{enp} Assume that
\begin{equation}\label{4.10}
\disp \alpha_i\;|D_y|\;|\rho_i|_{\infty}\;\; \mbox{ is small enough, for any } \; i=1, ..., n.
\end{equation}

Then $\mathfrak{L}$ is an invertible operator. It means, there exists a Nash equilibrium $\mathbf{h} = (h_1, . . ., h_n)$ for $J_i$.
\end{theorem}

\begin{proof} In the case $n=1$ we have $\mathcal{H}=\mathcal{H}_1$, $\mathbf{h}=h_1$ and therefore
\begin{equation}\label{4.11}
\begin{array}{l}
\disp \left(\mathfrak{L}\mathbf{h},\mathbf{h}\right)_{\mathcal{H}}=\left( |D_{y,s}|\; h_1 + \alpha_1 L^*_1 \left(\rho_1^2(y) |D_y| L_1 h_1 \right),h_1\right)_{\mathcal{H}_1}\\[5pt]\disp
=\left|\;|D_{y,s}|^{\frac 12}\; h_1\right|_{\mathcal{H}_1}^2+\alpha_1 \left| \rho_1(y)\; |D_y|^{\frac 12} L_1 h_1 \right|_{L^2(K)}^2\geq k_1|h_1|_{\mathcal{H}_1}^2,
\end{array}
\end{equation}
where we have used above the inequality (\ref{estd}). Then $\mathfrak{L}$ is coercive and by the Lax-Milgram Lemma, the equation
$\mathfrak{L}\mathbf{h} = \xi$ in $\mathcal{H}$, has a solution $\mathbf{h}$.

Now, suppose $n>1$. Again if we use (\ref{estd}) and the fact that  $L_i \in \mathcal{L}\left(\mathcal{H}_i, L^2(K)\right)$, one has
\begin{equation}\label{4.12}
\begin{array}{l}
\disp \left(\mathfrak{L}\mathbf{h},\mathbf{h}\right)_{\mathcal{H}}=\sum_{i=1}^n\left( |D_{y,s}|\; h_i,h_i\right)_{\mathcal{H}_i}+\sum_{i=1}^n\left( \alpha_i \rho_i\; |D_y| L_i h_i ,\rho_i\;L_ih_i\right)_{L^2(K)}\\[5pt]\disp
\geq\sum_{i=1}^n k_1|h_i|_{\mathcal{H}_i}^2+\sum_{i=1}^n\left( \alpha_i \rho_i\; |D_y| L_i h_i ,\rho_i\;L_ih_i\right)_{L^2(K)}.
\end{array}
\end{equation}

Next, it will be analyzed the second term on the right-hand side of (\ref{4.12}). In fact, taking the absolute value one gets
\begin{equation}\label{4.13}
\begin{array}{l}
\disp \left|\sum_{i=1}^n\left( \alpha_i \rho_i\; |D_y| L_i h_i ,\rho_i\;L_ih_i\right)_{L^2(K)}\right|\leq \overline{\alpha}\;k_4\;\overline{\rho}\sum_{i,j=1}^n|L_jh_j|_{L^2(K)}\;|L_ih_i|_{L^2(K)} \\[7pt]\disp
\leq \overline{\alpha}\;k_4\;\overline{\rho}\;\left[ \left(\sum_{i=1}^n|L_ih_i|_{L^2(K)}^2\right)^{\frac 12}(n)^{\frac 12}   \right]^2
\leq \overline{\alpha}\;k_4\;\overline{\rho}\;n\;C_S^2\sum_{i=1}^n|h_i|_{\mathcal{H}_i}^2,
\end{array}
\end{equation}
where $k_4$ is defined in (\ref{estd}), $\overline{\alpha}= \max\{\alpha_1, . . ., \alpha_n\}$, $C_S =\max\{C_i(S)\}$, with $C_i(S)$ defined in (\ref{cis}) and $\overline{\rho} = \max\{\rho_1, ..., \rho_n\}$. As for hypothesis $\disp \alpha_i\;|D_y|\;|\rho_i|_{\infty}$ is small enough, we assume
\begin{equation}\label{dmini}
\overline{\alpha}\;k_4\;\overline{\rho}\;n\;C_S^2 \leq \frac {k_1}{2}.
\end{equation}

Substituting (\ref{dmini}), into (\ref{4.13}) and using (\ref{4.12}) we finally obtain
\begin{equation}\label{4.14}
\begin{array}{l}
\disp \left(\mathfrak{L}\mathbf{h},\mathbf{h}\right)_{\mathcal{H}}\geq \frac {k_1}{2} \;|\mathbf{h}|_{\mathcal{H}}^2.
\end{array}
\end{equation}

Thus (\ref{4.11}) and (\ref{4.14}) say that Lax-Milgram Lemma is applicable and there exists only one $\mathbf{h}$ solution of $\mathfrak{L}\mathbf{h} = \xi$ in $\mathcal{H}$ for all $i = 1, . . .,n$.

\end{proof}

\section{Appendix: Existence, uniqueness and regularity of the state}\label{sec5}

In this Appendix, we recall a theoretical result for the linear system
\begin{equation}\label{a1}
\left| \begin{array}{l}
 \disp v_s + Lv+A(y,s)v+B(y,s).\nabla v - \frac{N}{2}\;v = \varphi(y,s) \;\mbox{ in }\; \mathbb{R}^N \times (0,S), \\
 v(y, 0) = v_0(y)  \mbox{ in } \mathbb{R}^N,
 \end{array}\right.
\end{equation}
where $A(y,s) \in L^\infty(\re^N\times(0,S))$, \,\, $B(y,s) \in (L^\infty(\re^N\times(0,S)))^N$.

We have the following:
\begin{theorem}\label{eul} Given $\varphi\in L^2(0,S;L^2(K))$ and $v_0 \in L^2(K)$, there exists a unique solution $v$ in the space $W\left(0, S, H^1(K), H^{-1}(K)\right)$ of the problem (\ref{a1}). Moreover, if $v_0 \in H^1(K)$, then $v \in W\left(0, S, H^2(K), L^2(K)\right)$.
\end{theorem}

 {\sc Sketch of the proof:} We will employ Galerkin method. In fact, let $\disp(w_i)_{i \in \mathbb{N}}$ be a Hilbertian basis of $H^1(K)$. Represent by $V_m$ the subspace generated by $\disp \{w_1, ..., w_n\}$ and let us consider the finite-dimensional problem

\begin{equation}\label{a2}
\left|
\begin{array}{l}
\disp \mbox{ Find }\; v_m \in V_m \;\;\mbox{ solution of }\\[7pt]\disp
(v_m',z)+(Lv_m,z)+(Av_m,z)+(B.\nabla v_m,z)-\frac N2 (v_m,z)=(\varphi,z) \;\; \mbox{ for all } \; z_m \in V_m \\[7pt]\disp
v_m(0) =v_{0m}\to v_0 \;\; \mbox{ strongly in } \;\; L^2(K)
\end{array}
\right.
\end{equation}

To obtain the first a priori estimate, we take $z=v_m$ in (\ref{a2}). We obtain:
\begin{equation}\label{a3}
\begin{array}{l}
\disp \frac 12 \frac{d}{ds} |v_m(s)|^2_{L^2(K)}+|v_m(s)|^2_{H^1(K)}\\[7pt]\disp
\leq A_0|v_m(s)|^2_{L^2(K)}+ B_0 |v_m(s)|_{H^1(K)} |v_m(s)|_{L^2(K)}+|\varphi(s)|_{L^2(K)}|v_m(s)|_{L^2(K)},
\end{array}
\end{equation}
where $A_0=|A|_{L^\infty(\re^N\times(0,S))}$, and $B_0=|B|_{(L^\infty(\re^N\times(0,S)))^N}$.

Hence by integration from $0$ to $s\leq S$, in the local interval of existence of $v_m(s)$, and using Young's inequalities and, again, the Gronwall's inequality yields
\begin{equation}\label{a4}
\begin{array}{l}
\disp |v_m(s)|_{L^2(K)}\leq C\;\;\; \mbox{ and } \;\;\;\int_0^S|v_m(s)|^2_{H^1(K)}ds \leq C,
\end{array}
\end{equation}
after the extension of $v_m(s)$ to $[0,S]$.

With the estimates (\ref{a4}) we take the limit in (\ref{a2}) obtaining a function $v$ in $W\left(0, S, H^1(K), H^{-1}(K)\right)$ which is weak solution of (\ref{a1}). We also have uniqueness.

We furthermore suppose that $v_0 \in H^1(K)$. In the approximate system (\ref{a2}) we suppose $v_{0m} \to v_0$ strongly in $H^1(K)$. Then we can obtain a second a priori estimate. We take $z=v'_m(t)$ in (\ref{a2}). We obtain:
$$
\aligned
|v_m'(s)|_{L^2(K)}^2 &+ \frac 12\,\frac{d}{ds}\,|v_m(s)|_{H^1(K)}^2 \le A_0C_0|v_m(s)|_{H^1(K)}\,|v_m'(s)|_{L^2(K)} \,\\[7pt]
&+ B_0|v_m(s)|_{H^1(K)}\,|v_m'(s)|_{L^2(K)} + |\varphi(s)|_{L^2(K)}\,|v_m'(s)|_{L^2(K)},
\endaligned
$$
where we have used the Poincar\'e inequality, i.e., $|v_m(s)|_{L^2(K)}\leq C_0|v_m(s)|_{H^1(K)}$ (see $(\ref{desigualdades})_2)$.

Thus,
$$
\frac 14\,|v_m'(s)|_{L^2(K)}^2 + \frac{d}{ds}\,|v_m(s)|_{H^1(K)}^2 \le 2|\varphi(s)|_{L^2(K)}^2 + 2(A_0^2C_0^2 + B_0^2)|v_m(s)|_{H^1(K)}^2.
$$

By Gronwall's lemma we obtain:
\begin{equation}\label{a5}
\begin{array}{l}
\disp \int_0^S|v_m'(s)|_{L^2(K)}ds\leq C\;\;\; \mbox{ and } \;\;\;|v_m(s)|_{H^1(K)} \leq C, \;\;\mbox{ on } \; 0\leq s \leq S.
\end{array}
\end{equation}

We obtain a strong solution by the process of regularity of an elliptic equation. We obtain a strong solution $v \in W\left(0, S, H^2(K), L^2(K)\right)\subset C([0,T], H^1(K))$, by (\ref{4imersoes}).

\Fin

Next, we give a precise definition of (\ref{sa}) in the sense of transposition and prove an existence and uniqueness result. The main question we are concerned here consists in finding a solution $p$ of the of the parabolic problem:
\begin{equation}\label{a6}
\left|
\begin{array}{l}
\disp -p_s + Lp + Ap - \mbox{div}(Bp) - \frac N2\,p - \frac 12\, y\;.\;Bp=0\; \mbox{ in } \;\re^N \times (0,S),\\[5pt]\disp
p(y,S) = \xi \quad\mbox{ in }\quad \re^N.
\end{array} \right.
\end{equation}
when $\xi \in H^{-1}(K)$. A function $p \in L^2(0,S;L^2(K))$ is called an solution of (\ref{a6}) obtained by transposition if
$$
\int_0^S \int_{\re^N} p(y,s)\varphi(y,s)K(y)\,dyds = \langle \xi,v(S)\rangle
$$
for any solution $v$ of the problem (\ref{a1}) with $\varphi \in L^2(0,S;L^2(K))$ and $v_0(y)=0$ in $\re^N$. We represent by $\disp \langle \;, \;\rangle$ the duality pairing between $H^{-1}(K)$ and $H^{1}(K)$. We have

\begin{theorem}\label{ta2} If $\xi \in H^{-1}(K)$, $A(y,s) \in L^\infty(\re^N\times(0,S))$,  $B(y,s) \in (L^\infty(\re^N\times(0,S)))^N$ then there exists an unique solution $p \in L^2(0,S;L^2(K))\cap C\left([0,S];H^{-1}(K)\right)$ of problem (\ref{a6}).

\end{theorem}

\begin{proof} Let us consider the linear form $ F\colon L^2(0,S;L^2(K)) \to \re$ defined by
\begin{equation}\label{dF}
\disp F(\varphi) = \langle \xi,v(S)\rangle,
\end{equation}
for all $\varphi \in L^2(0,S;L^2(K))$, where  $v$ is the solution of (\ref{a1}), with $v_0=0$, corresponding to $\varphi$.

Since $v_0=0$, it is easy see that
\begin{equation}\label{a7}
\disp|v(S)|_{H^1(K)} \leq C|\varphi|_{L^2(0,S;L^2(K))}.
\end{equation}

Thus, $F$ is a continuous linear form in $L^2(0,S;L^2(K))$. By Riesz's representation theorem, there exists a unique $p\in L^2(0,S;L^2(K))$ such that
\begin{equation}\label{a8}
\disp F(\varphi)=\int_0^S \int_{\re^N} p\;\varphi\;K\,dyds, \;\; \mbox{ for all } \;\; \varphi \in L^2(0,S;L^2(K)).
\end{equation}

The uniqueness is consequence of the Du Bois Raymond Lemma. Moreover,
\begin{equation}\label{a9}
\disp |F|_{L^2(0,S;L^2(K))}=|p|_{L^2(0,S;L^2(K))}.
\end{equation}

Now, due to (\ref{dF}), (\ref{a7}) and (\ref{a9}), for some constant $C$, we have
\begin{equation}\label{a10}
\disp |p|_{L^2(0,S;L^2(K))}\leq C \;|\xi|_{H^{-1}(K)}.
\end{equation}

Since $\xi \in H^{-1}(K)$, there exists $\disp (\xi_m)_{m\in \mathbb{N}}$ with $\xi_m \in L^2(K)$ for all $m$, such that
\begin{equation}\label{a11}
\disp \xi_m \to \xi \quad\mbox{strongly in }\quad H^{-1}(K).
\end{equation}

Let $(p_m)$ be the sequence of solutions of (\ref{a6}) corresponding to $\xi_m$ for each $m$. Obviously, the function $p_n-p_m$ is the solution, by transposition, of (\ref{a6}) corresponding to $\xi_n-\xi_m$.

From (\ref{a10}) and (\ref{a11}) we have that $\disp (p_m)_{m\in \mathbb{N}}$ is a Cauchy sequence in $L^2(0,S;L^2(K))$ and
\begin{equation}\label{a12}
\disp p_m \to {p} \quad\mbox{ strongly in }\quad L^2(0,S,L^2(K)).
\end{equation}

On the other hand,
$$
\aligned
&\disp |\langle p_m'- p_n',\psi\rangle| \leq |\langle p_m-p_n, L\psi\rangle| + |\langle A(p_m-p_n),\psi\rangle|\,\\
&\disp + |\langle p_m-p_n, B\;.\;\nabla\psi\rangle| + \frac N2\,|\langle p_m-p_n,\psi\rangle|\,\\
&\disp \leq C|p_m-p_n|_{L^2(0,S,L^2(K))}\,|\psi|_{L^2(0,S,H^2(K))}\,,
\endaligned
$$
for all $\psi \in L^2(0,S;H^2(K))$, and this implies that $\disp (p_m')_{m\in \mathbb{N}}$ is a Cauchy sequence in $L^2(0,S;H^{-2}(K))$.

Therefore, we have
\begin{equation}\label{a13}
\disp p_m \to {p} \quad\mbox{ strongly in }\quad W(0, S, L^2(K), H^{-2}(K)),
\end{equation}
and this space is continuous embedded in $C([0,S],H^{-1}(K))$.

We obtain $\disp p_m \to p$ strongly in $ C([0,S],H^{-1}(K))$. In particular, $p \in C([0,S],H^{-1}(K))$. This completes the proof of Theorem \ref{ta2}.

\end{proof}

\end{document}